\documentclass[12pt]{article}
\usepackage{color}
\usepackage{graphicx}
\usepackage{colortbl}
\usepackage{array}
\usepackage{amssymb}
\usepackage{amsmath}
\usepackage{mathrsfs}
\usepackage{dcolumn}
\usepackage{longtable}
\usepackage{hhline}
\usepackage{graphics}
\usepackage{amssymb}
\usepackage{amsmath}
\usepackage{amsthm}
\usepackage{mathrsfs}
\usepackage{amsfonts}
\usepackage{latexsym}
\usepackage{color}
\usepackage{epsf}

\newtheorem{thm}{Theorem}[section]

\newtheorem{rema}[thm]{Remark}

\title{\bf On Divisibility Property of Type 2 $(p,q)$-Analogue of $r$-Whitney Numbers\\ of the Second Kind}
\author{\large{\bf Roberto B. Corcino}\\
{\bf Cristina B. Corcino}\\
Research Institute for Computational\\ 
Mathematics and Physics\\
	Cebu Normal University\\
	Cebu City, Philippines 6000\\
	{\tt rcorcino@yahoo.com}  
}

\begin{document}

\maketitle

\begin{abstract}
	In this paper, the divisibility property of the type 2 $(p,q)$-analogue of the $r$-Whitney numbers of the second kind is established. More precisely, a congruence relation modulo $pq$ for this $(p,q)$-analogue is derived.

\end{abstract}

\section{Introduction}
The $r$-Whitney numbers of the second kind were introduced by Mezo \cite{mezo} as coefficients of the following generating function:
 \begin{equation*}
(mx+r)^n=\sum_{k=0}^nm^kW_{m,r}(n,k)x^{\underline{k}},
\end{equation*}
where $x^{\underline{k}} = x(x-1)\ldots (x-k+1)$. These numbers satisfy the following properties:
\begin{enumerate}
\item the exponential generating function
$$\sum_{n=0}^{\infty}W_{m,r}(n,k)\frac{z^n}{n!}=\frac{e^{rz}}{k!}\left(\frac{e^{mz}-1}{m}\right)^k,$$
\item the explicit formula
$$W_{m,r}(n,k)=\frac{1}{m^kk!}\sum_{i=0}^k\binom{k}{i}(-1)^{k-i}(mi+r)^n,$$
\item the triangular recurrence relation
$$W_{m,r}(n,k)=W_{m,r}(n-1,k-1)+(km+r)W_{m,r}(n-1,k).$$
\end{enumerate}
These properties are exactly the same properties that the $(r,\beta)$-Stirling numbers in \cite{Cor-2} have possessed. This implies that the $r$-Whitney numbers of the second kind and the $(r,\beta)$-Stirling numbers are equivalent. More properties of these numbers can be found in \cite{Cor-2, mezo, CORCA}.

\smallskip
One of the early studies on $q$-analogue of Stirling numbers of the second kind was introduced by Carlitz in \cite{Car1} in connection with a problem in abelian groups. This is known as $q$-Stirling numbers of the second kind and is defined in terms of the following recurrence relation
$$S_q[n,k]=S_q[n-1,k-1]+[k]_qS_q[n-1,k], \;\;\;[k]_q=\frac{1-q^k}{1-q}$$
such that, when $q\to1$, this gives the triangular recurrence relation for the classical Stirling numbers of the second kind $S(n,k)$
$$S(n,k)=S(n-1,k-1)+kS(n-1,k).$$
Another version of definition of this $q$-analogue was adapted in \cite{Ehr} as follows
\begin{equation}\label{qSNEhr}
S_q[n,k]=q^{k-1}S_q[n-1,k-1]+[k]_qS_q[n-1,k].
\end{equation}
Through this definition, the Hankel transform of $q$-exponential polynomials and numbers was successfully established, which may be considered as the Hankel transform of a certain $q$-analogue of Bell polynomials and numbers. 

\smallskip
There are many ways to define $q$-analogue of Stirling-type and Bell-type numbers (see \cite{CORCHT, Corcino1, jen, rb4, rb3, rb5}). However, in the desire to establish the Hankel transform of $q$-analogue of generalized Bell numbers, Corcino et al. \cite{joy} were motivated to define a $q$-analogue of $r$-Whitney numbers of the second kind parallel to that in (\ref{qSNEhr}) as follows:
\begin{equation}\label{qrDN}
W_{m,r}[n,k]_{q}= q^{m(k-1)-r}W_{m,r}[n-1,k-1]_q+[mk-r]_{q} W_{m,r}[n-1,k]_{q}.
\end{equation}
Two more forms of this $q$-analogue, denoted by $W^*_{m,r}[n,k]_q$ and $\widetilde{W}_{m,r}[n,k]_q$, were respectively defined by
\begin{align*}
W^*_{m,r}[n,k]_q&:=q^{-kr+m\binom{k}{2}}W_{m,r}[n,k]_q,\\
\widetilde{W}_{m,r}[n,k]_q&:= q^{-kr}W^*_{m,r}[n,k]_q=q^{-m\binom{k}{2}}W_{m,r}[n,k].
\end{align*}
The corresponding $q$-analogues of generalized Bell numbers, also known as $q$-analogues of $r$-Dowling numbers, were also defined in three forms as (see  \cite{cor2, joy, cor})
\begin{align*}
{D}_{m,r}[n]_q&:=\sum_{k=0}^nW_{m,r}[n,k]_q,\\
{D}^*_{m,r}[n]_q&:=\sum_{k=0}^nW^*_{m,r}[n,k]_q,
\end{align*}
and
\begin{equation*}
\widetilde{D}_{m,r}[n]_q:=\sum_{k=0}^n\widetilde{W}_{m,r}[n,k]_q.
\end{equation*}
where ${D}_{m,r}[n]_q, {D}^*_{m,r}[n]_q$ and $\widetilde{D}_{m,r}[n]_q$ denote the first, second and third form of the $q$-analogues of $r$-Dowling numbers, respectively. The Hankel transforms of ${D}_{m,r}[n]_q, {D}^*_{m,r}[n]_q$ and $\widetilde{D}_{m,r}[n]_q$ were successfully established in \cite{cor, cor2, joy}.

\smallskip
To extend these research studies, a certain $(p,q)$-analogue of $r$-Whitney numbers of the second kind, denoted by $W_{m,r}[n,k]_{p,q}$, was defined in \cite{ame} as coefficients of the following generating function:
 \begin{equation}
[mt+r]_{p,q}^n=\sum_{k=0}^n W_{m,r}[n,k]_{p,q}[m t|m]_{p,q}^{\underline{k}}\label{def2}
\end{equation}	
where
\begin{equation}
[t|m]_{p,q}^{\underline{n}}=\prod_{j=0}^{n-1}[t-jm]_{p,q}.
\end{equation}
The orthogonality and inverse relations, an explicit formula, and a kind of exponential generating function of $W_{m,r}[n,k]_{p,q}$ were already obtained. Unfortunately, its Hankel transform was not successfully established using the method applied in \cite{cor, cor2, joy}. This motivated Corcino et al. \cite{ame-1} to define the type 2 $(p,q)$-analogue of $r$-Whitney numbers of the second kind, denoted by $W_{m,r}[n,k;t]_{p,q}$, as follows: 
\begin{equation}\label{tri2}
W_{m,r}[n+1,k;t]_{p,q}=q^{m(k-1)+r}W_{m,r}[n,k-1;t]_{p,q}+[mk+r]_{p,q}p^{mt-km}W_{m,r}[n,k;t]_{p,q}.
\end{equation}
The second form was then defined as follows:
\begin{equation}\label{def2ndform}
W_{m,r}^*[n,k;t]_{p,q}:=q^{-kr-m\binom{k}{2}}W_{m,r}[n,k;t]_{p,q}.
\end{equation}
Several properties of these $(p,q)$-analogues were established in \cite{ame-1} including their Hankel transforms, which are given by
\begin{align*}
\det{(W_{m,r}[s+i+j,s+j;t]_{p,q})_{0\leq i,j \leq n}}&=\prod_{k=0}^nq^{m\binom{s+k}{2}+(s+k)r}p^{nmt}[m(s+k)+r]_{p,q}^k\\
\det(W^*_{m,r}[s+i+j,s+j;t]_{p,q})_{0\leq i,j \leq n}&=\prod_{k=0}^np^{nmt}[m(s+k)+r]^k_{p,q}.
\end{align*}
On the other hand, the first, second and third forms of type 2 $(p,q)$-analogue of the $r$-Dowling numbers, denoted by $D_{m,r}[n]_{p,q}$, $D^*_{m,r}[n]_{p,q}$ and $\widetilde{D}_{m,r}[n]_{p,q}$ were defined respectively in \cite{ame-1} as follows:
\begin{align*}
	D_{m,r}[n]_{p,q}&:=\sum_{k=0}^nW_{m,r}[n,k;t]_{p,q},\\
	D^*_{m,r}[n]_{p,q}&:=\sum_{k=0}^nW^*_{m,r}[n,k;t]_{p,q},\\
	\widetilde{D}_{m,r}[n]_{p,q}&:=\sum_{k=0}^n\widetilde{W}_{m,r}[n,k;t]_{p,q},
\end{align*}
where $\widetilde{W}_{m,r}[n,k;t]_{p,q}=q^{kr}W^*_{m,r}[n,k;t]_{p,q}$ denotes the third form of the $(p,q)$-analogue of the $r$-Whitney numbers of the second kind. Among these three forms, only the second form was provided a Hankel transform, which is given by
\begin{equation*}
H(D^*_{m,r}[n]_{p,q})=\left(\frac{q}{p}\right)^{\frac{n(n^2+3n+8)}{6}+r-1\binom{n}{2}}([m]_{\frac{q}{p}})^{\binom{n}{2}}\prod_{k=0}^{n-1}[k]_{\left(\frac{q}{p}\right)^m}!.
\end{equation*}

\smallskip
The main objective of this study is to establish additional property of the type 2 $(p,q)$-analogues of the $r$-Whitney numbers of the second kind. More precisely, the divisibility property of these type 2 $(p,q)$-analogues will be discussed thoroughly.

\section{Preliminary Results}
This section provides a brief discussion on some relations that are necessary in deriving the divisibility property of the type 2 $(p,q)$-analogue of the $r$-Whitney numbers of the second kind $W^*_{m,r}[n,k;t]_{p,q}$.

\smallskip
Multiplying both sides of the recurrence relation in \eqref{tri2} by $q^{-kr-m\binom{k}{2}}$ yields
\begin{align*}
q^{-kr-m\binom{k}{2}}W_{m,r}[n+1,k;t]_{p,q}&=q^{-kr-m\binom{k}{2}}q^{m(k-1)+r}W_{m,r}[n,k-1;t]_{p,q}\\
&\;\;\;\;+q^{-kr-m\binom{k}{2}}[mk+r]_{p,q}p^{mt-km}W_{m,r}[n,k;t]_{p,q}\\
q^{-kr-m\binom{k}{2}}W_{m,r}[n+1,k;t]_{p,q}&=q^{-(k-1)r-m\binom{k-1}{2}}W_{m,r}[n,k-1;t]_{p,q}\\
&\;\;\;\;+[mk+r]_{p,q}p^{mt-km}q^{-kr-m\binom{k}{2}}W_{m,r}[n,k;t]_{p,q}.
\end{align*}
Applying \eqref{def2ndform} consequently gives
\begin{equation}\label{rec-2ndform}
W^*_{m,r}[n+1,k;t]_{p,q}=W^*_{m,r}[n,k-1;t]_{p,q}+[mk+r]_{p,q}p^{mt-km}W^*_{m,r}[n,k;t]_{p,q}.
\end{equation}
This relation can be used to generate the following first few values of $W^*_{m,r}[n,k;t]_{p,q}$:
\begin{center}
\begin{tabular}{|c|c|c|c|c|}
\hline $n/k$ & $0$ & $1$ & $2$ & $3$\\
\hline $0$ & $1$ & & &\\
\hline $1$ & $[r]_{p,q}p^{mt}$ & $1$ & &\\
\hline $2$ & $[r]_{p,q}^{2}p^{2mt}$ & $[r]_{p,q}p^{mt}+[m+r]_{p,q}p^{m(t-1)}$ & $1$ &\\
\hline $3$ & $[r]_{p,q}^{3}p^{3mt}$ & $[r]_{p,q}^{2}p^{2mt}+[r]_{p,q}[m+r]_{p,q}p^{m(2t-1)}$ & $[r]_{p,q}p^{mt}+2[m+r]_{p,q}p^{m(t-1)}$ & $1$\\
\hline $$ & $$ & $+[m+r]_{p,q}^{2}p^{2m(t-1)}$ &  & \\
\hline
\end{tabular}

\smallskip
{Table 1. The First Values of $W^*_{m,r}[n,k;t]_{p,q}$}
\end{center} 

By repeated application of \eqref{rec-2ndform}, we can easily derive the following vertical recurrence relation.
\begin{thm}
For nonnegative integers $n$ and $k$, and real number $r$, the  $(p,q)$-analogue  of $r$-Whitney numbers of the second kind satisfies the following vertical recurrence relation
\begin{equation}\label{vrr}
W^*_{m,r}[n+1,k+1;t]_{p,q}=\sum_{j=k}^n [m(k+1)+r]_{p,q}^{n-j}p^{(n-j)[mt-(k+1)m]}W^*_{m,r}[j,k;t]_{p,q}.
\end{equation}
\end{thm}
\noindent One can easily verify relation \eqref{vrr} using the values of $W^*_{m,r}[n,k;t]_{p,q}$ in Table 1.

\bigskip
Now, let us derive the rational generating function for $W^*_{m,r}[n,k;t]_{p,q}$. Suppose that 
$$\Psi^*_k(x)=\sum_{n=k}^{\infty}W^*_{m,r}[n,k;t]_{p,q}x^{n-k}.$$
When $k=0$, \eqref{rec-2ndform} reduces to
$$W^*_{m,r}[n+1,0;t]_{p,q}=[r]_{p,q}p^{mt}W^*_{m,r}[n,0;t]_{p,q}.$$
By repeated application of \eqref{rec-2ndform}, this inductively gives
\begin{align*}
W^*_{m,r}[n+1,0;t]_{p,q}&=[r]_{p,q}p^{mt}W^*_{m,r}[n,0;t]_{p,q}=\left([r]_{p,q}p^{mt}\right)^2W^*_{m,r}[n-1,0;t]_{p,q}\\
&\vdots\\
&=\left([r]_{p,q}p^{mt}\right)^{n+1}W^*_{m,r}[0,0;t]_{p,q}=\left([r]_{p,q}p^{mt}\right)^{n+1}.
\end{align*}
Hence,
$$\Psi^*_0(x)=\sum_{n=0}^{\infty}W^*_{m,r}[n,0;t]_{p,q}x^n=\frac{1}{(1-xp^{mt}[r]_{p,q})}.$$
When $k>0$ and applying the triangular recurrence relation in \eqref{tri2}, we have
	\begin{align*}
	\Psi^*_k(x)&=\sum_{n=k}^{\infty}W^*_{m,r}[n,k;t]_{p,q}x^{n-k} \\
	&=\sum_{n-1=k-1}^{\infty}W^*_{m,r}[n-1,k-1;t]_{p,q}x^{(n-1)(k-1)}\\
	& \;\;\;\;\;\;\;\;\;+xp^{mt-km}[mk+r]_{p,q}\sum_{n-1=k}^{\infty}W^*_{m,r}[n-1,k;t]_{p,q}x^{n-1-k}\\
	&=\Psi^*_{k-1}(x)+xp^{m(t-k)}[mk+r]_{p,q}\Psi^*_k(x)
	\end{align*}
Solving for $\Psi^*_k(t)$ yields
$$\Psi^*_k(x)=\frac{1}{1-xp^{m(t-k)}[mk+r]_{p,q}}\Psi^*_{k-1}(x).$$
Applying backward substitution gives the following rational generating function for $W_{m,r}[n,k;t]_{p,q}$.
	
\begin{thm} \label{thmrat}For nonnegative integers $n$ and $k$, and real number $r$, the  $(p,q)$-analogue $W_{m,r}[n,k;t]_{p,q}$ satisfies the following rational generating
	function
	\begin{equation}\label{Wpqrat}
	\Psi^*_k(x)=\sum_{n=k}^{\infty}W^*_{m,r}[n,k;t]_{p,q}x^{n-k}=\frac{1}{\prod_{j=0}^{k}(1-xp^{m(t-j)}[mj+r]_{p,q})}.
	\end{equation}
\end{thm}
\begin{rema}
This rational generating function plays an important role in proving the main result of the paper.
\end{rema}

\section{Divisibility Property}

In this section, the congruence relation modulo $pq$ for the type 2 $(p,q)$-analogue of the $r$-Whitney numbers of the second kind $W^*_{m,r}[n,k;t]_{p,q}$ will be established using the rational generating function in \eqref{Wpqrat}. 

\smallskip
Using the values of $W^*_{m,r}[n,k;t]_{p,q}$ in Table 1, we observe that, with 
$$[t]_{p,q}=p^{t-1}+p^{t-2}q+p^{t-3}q^2+\ldots +pq^{t-2}+q^{t-1},$$
the polynomial expressions of $W_{m,r}^{*}[n,k]_{q}$ from row 0 to row 3, if they are divided by $pq$, the remainders form the following triangle of expressions in $p$:
\begin{center}
$1$\\
$p^{mt+r-1}\;\;\;\;\;\; 1$\\
$p^{2(mt+r-1)}\;\;\;\;\;\; 2p^{mt+r-1}\;\;\;\;\;\; 1$\\
$p^{3(mt+r-1)}\;\;\;\;\;\; 3p^{2(mt+r-1)}\;\;\;\;\;\; 3p^{mt+r-1}\;\;\;\;\;\; 1$. 
\end{center}
This can further be written as 
\begin{center}
$\binom{0}{0}$\\
$\binom{1}{0}p^{mt+r-1}\;\;\;\;\;\; \binom{1}{1}$\\
$\binom{2}{0}p^{2(mt+r-1)}\;\;\;\;\;\; \binom{2}{1}p^{mt+r-1}\;\;\;\;\;\; \binom{2}{2}$\\
$\binom{3}{0}p^{3(mt+r-1)}\;\;\;\;\;\; \binom{3}{1}p^{2(mt+r-1)}\;\;\;\;\;\; \binom{3}{2}p^{mt+r-1}\;\;\;\;\;\; \binom{3}{3}$, 
\end{center}
To generalize this observation, the next theorem contains the divisibility property of $W^*_{m,r}[n,k;t]_{p,q}$.

\begin{thm}\label{div}
For nonnegative integers $n$ and $k$, the type 2 $(p,q)$-analogue of the $r$-Whitney numbers of the second kind $W_{m,r}[n,k;t]_{p,q}$  satisfies the following congruence relation
\begin{equation}\label{pqrWhitney_div}
W^*_{m,r}[n,k;t]_{p,q}\equiv\binom{n}{k}p^{(n-k)(mt+r-1)} \mod pq.
\end{equation}
\end{thm}
\begin{proof}
The polynomial $[t]_{p,q}$ can be written as
$$[t]_{p,q}=p^{t-1}+q^{t-1}+pqy,$$
where $y$ is a polynomial in $p$ and $q$. Then, we have
\begin{align*}
\frac{1}{\prod_{j=0}^{k}(1-xp^{m(t-j)}[mj+r]_{p,q})}&=\sum_{n=0}^{\infty}\left(xp^{m(t-j)}[mj+r]_{p,q}\right)^n\\
&=\sum_{n=0}^{\infty}p^{nm(t-j)}(p^{mj+r-1}+q^{mj+r-1}+pqy)^nx^n\\
&=\sum_{n=0}^{\infty}p^{n(mt+r-j)}x^n+pq\sum_{n=0}^{\infty}\hat{z}_nx^n,
\end{align*}
where $\hat{z}_n$ is a polynomial in $p$ and $q$. It follows that
\begin{align*}
\frac{1}{\prod_{j=0}^{k}(1-xp^{m(t-j)}[mj+r]_{p,q})}&=\sum_{n=0}^{\infty}p^{n(mt+r-j)}x^n \mod pq\\
&=\frac{1}{1-p^{mt+r-1}x} \mod pq.
\end{align*}
Thus, using \eqref{Wpqrat}, we have
\begin{align*}
\sum_{n=k}^{\infty}W^*_{m,r}[n,k;t]_{p,q}x^{n-k}&\equiv\frac{1}{(1-p^{mt+r-1}x)^{k+1}} \mod pq\\
&\equiv \sum_{n=0}^{\infty} \binom{n+(k+1)-1}{n}p^{n(mt+r-1)}x^n  \mod pq\\
&\equiv \sum_{n=k}^{\infty} \binom{n}{k}p^{(n-k)(mt+r-1)}x^{n-k}  \mod pq.
\end{align*}
Comparing the coefficients of $x^{n-k}$ completes the proof of the theorem.
\end{proof}
\begin{rema}
Using \eqref{def2ndform} and Theorem \ref{div}, the first form of the of the type 2 $(p,q)$-analogues of the $r$-Whitney numbers of the second kind satisfies the following congruence relation modulo $pq$:
\begin{align*}
W_{m,r}[n,k;t]_{p,q}&\equiv \binom{n}{k}p^{(n-k)(mt+r-1)}q^{kr+m\binom{k}{2}} \mod pq\\
&\equiv 
\begin{cases}
    q^{nr+m\binom{n}{2}} \mod pq, & \text{for } n=k \\
    0 \mod pq, & \text{otherwise } 
\end{cases}
\end{align*}
\end{rema}

\end{document}